\documentclass[11pt]{amsart}
\usepackage{epsfig,graphicx, subfigure,epstopdf}
\usepackage{amsmath, amssymb,mathabx,mathrsfs}
\usepackage{hyperref}
\usepackage{color}
\usepackage[toc,page]{appendix}
\usepackage{comment}

\newtheorem{thm}{Theorem}[section]
\newtheorem{cor}[thm]{Corollary}

\newtheorem{prop}[thm]{Proposition}
\theoremstyle{definition}
\newtheorem{defn}[thm]{Definition}
\theoremstyle{remark}
\newtheorem{rem}[thm]{Remark}
\newtheorem{ex}[thm]{Example}
\numberwithin{equation}{section}
 
\theoremstyle{definition}

\newcommand{\inte}{\textrm{int}}

\def\R{{\mathbb R}}

	\title[Proof of Kakutani's Fixed Point Theorem]{Combinatorial Proof of Kakutani's Fixed Point Theorem}

	\author[Yitzchak Shmalo]{Yitzchak Shmalo}
	
	 \email{yitzchak.shmalo@mail.yu.edu}
	

\thanks{This research was partially supported by NSF grant DMS-1700154.}

\begin{document}
	
		\maketitle

	

\begin{abstract}
Kakutani's fixed point theorem is a generalization of Brouwer's fixed point theorem to upper semicontinuous multivalued maps and is used extensively in game theory and other areas of economics. Earlier works  have shown that Sperner's lemma implies Brouwer's theorem. In this paper, a new combinatorial labeling lemma, generalizing Sperner's original lemma, is given and is used to derive a simple proof for Kakutani's fixed point theorem. The proof is constructive and  can be easily applied to numerically approximate the location of fixed points. The main method of the proof is also used to obtain a  generalization of Kakutani's theorem for discontinuous maps which are locally gross direction preserving.
\end{abstract}


	\section{Introduction}

It is well known that Brouwer's fixed point theorem, which states that every continuous map from an $n$-dimensional closed ball to itself has a fixed point, can be proved using Sperner's lemma. This lemma asserts the existence of completely labeled simplices -- simplices that carry all labels from a prescribed collection of labels -- in a triangulation of a simplex when certain labeling conditions are satisfied.
When Sperner's lemma is used in the context of Brouwer's fixed point theorem, one can produce a labeling  that encodes information on the continuous map, and which satisfies the conditions required by Sperner's lemma.
This labeling can be used to prove the existence of fixed points, and to find numerical approximations of them, up
to any desired margin of error. 
More precisely,  completely labeled simplices in a triangulation of a simplex provide approximations of fixed points and prove their existence.
 
Kakutani's fixed point theorem is a generalization of Brouwer's theorem, for the case of  upper semicontinuous multivalued maps; however, it is more challenging to prove Kakutani's theorem from Sperner's lemma, as was recently shown \cite{LM}. The difficulty in using Sperner's lemma to prove Kakutani's theorem lies in the fact that,  in the case of multivalued upper semicontinuous maps, completely labeled simplices do not provide enough information for proving the existence of fixed points or approximating their location. In \cite{LM}, the author tried various ways of encoding the information on multivalued upper semicontinuous maps so that completely labeled simplices approximate fixed points, however, those approaches do not seem to work.

The original proof given by Kakutani \cite{KF} used Brouwer's fixed point theorem to argue for the existence of a sequence of fixed points of certain single-valued maps constructed out of the multivalued map. These sequences converge to the desired fixed point. Thus, Sperner's lemma can be used to find the fixed points of single-valued maps and they, in turn, can be used to find the fixed point of the multivalued map. Alternatively, we can approximate an upper semicontinuous multivalued map  with single-valued continuous functions; such approximations are shown to exist by  von Neumann's approximation lemma \cite{von1971model}. The fixed points of the single-valued maps can be shown to approximate the fixed points of the multivalued maps. Thus, using Sperner's lemma to approximate the location of  fixed points of  single valued maps we can prove the existence of   fixed points of multivalued maps. The technicalities of proving Kakutani's theorem directly from Sperner's lemma, using some variant of the methods described above, are given in \cite{LM,YT}, but both proofs have their drawbacks. In essence, both proofs require the approximation of an upper semicontinuous  multivalued map by a continuous single valued map, which in practice can be a complicated thing to do. They also indirectly invoke Brouwer's fixed point theorem, and using either method to numerically find fixed points would be computationally expensive. In short, the proofs rely on Sperner's lemma only indirectly, and do not achieve that completely labeled simplices are approximations of fixed points of upper semicontinuous multivalued  maps.

This paper provides  a new labeling lemma, which is a generalization of Sperner's original lemma,  in order to prove Kakutani's theorem and other generalizations of the theorem. In Section \ref{pre}, we introduce Sperner's lemma and Brouwer's fixed point theorem. In Section \ref{cubicle} we introduce a couple generalizations of Sperner's lemma, based on the Brouwer degree of maps, including a generalization which we call the hyperplane labeling lemma.
In Section \ref{kakutani}, the hyperplane labeling lemma is used to prove Kakutani's theorem in the case of a mapping from an $n$-dimensional cube to its power set. The proof is done using Cartesian coordinates. Also, a generalization of Kakutani's theorems is given, using the concept of locally gross direction preserving maps, which are maps that can be discontinuous in the normal sense. We prove that such maps have the fixed point property. Thus, the the hyperplane labeling lemma provides a combinatorial proof for the known result that every locally gross direction preserving map has a fixed point. For more information on the topic, see \cite{JEANJACQUESHERINGS200889, CROMME19971527, He2004, Termwuttipong2010}.   

The key idea of the hyperplane labeling lemma is that, unlike Sperner's lemma, it requires certain conditions for the labeling inside of a polytope, and not only on the boundary. This seems unavoidable when dealing with Kakutani's theorem. More specifically, we take a multivalued map of a $d$-dimensional polytope to itself. We then decompose the polytope into smaller polytopes, and label the vertices of the smaller polytopes. The labeling encodes information regarding the multivalued map. Because of the conditions for the labeling of vertices in the interior of the polytope, more information of the map is, essentially, preserved by the labeling. Thus, unlike with the Sperner lemma, every completely labeled small polytope in the decomposition  can be used to  prove the existence of a fixed point. As mentioned, this proof does not rely on Brouwer's fixed point theorem and does not require approximating the multivalued map by a single-valued map. 

In applications, multivalued maps have been used in areas of game theory, optimal control and partial deferential equations and recently they have also been used as a numerical tool for approximating single-valued maps, see \cite{Kaczynski2008}. Such numerical tools are helpful in the areas of dynamical systems and provide computer assisted proofs for problems concerning single-valued maps. Often in chaos theory it is complicated to try and evaluate a single-valued function perfectly and therefore it might be easier to study an arbitrarily close approximation of the single-valued map via a multivalued map, see \cite{10.2307/2585169,Kaczynski2008}. Thus, it is convenient to have a simple combinatorial proof of Kakutani's fixed point theorem, one which can numerically approximate the location of fixed points.

	\section{Preliminaries}\label{pre}
	
	\begin{defn}[Sperner Labeling]
		Consider the standard $n$-dimensional simplex $T=\textrm{conv}(0, e_1,e_2,\ldots ,e_n)\subset \mathbb{R}^n$, where we denote by $(e_1,e_2,\ldots ,e_{n})$ the standard basis of $\mathbb{R}^n$, by $\textrm{conv}(\cdot)$ the convex hull of a set and by $\{v_0,v_1,v_2,\ldots ,v_{n}\}$ the vertices of $T$.  A triangulation $\{T_i\}_{i\in I}$ of the simplex is a subdivision of the simplex into smaller simplices such that for any two simplices in the subdivision,
$T_i$, $T_j$, the intersection $T_i\cap T_j$ is either empty or a face of both $T_i$ and $T_j$. If $K$ is a subset of $T$, the carrier $\textrm{car}(K)$ of $K$ is the smallest face of $T$  that contains $K$.	

		A Sperner labeling of the triangulation is a function $\phi$ that assigns to each vertex $v$ of  the triangulation a number $\phi(x)\in\{0,1,\ldots n\}$ such that:
		\begin{itemize}
			\item [(i)] each vertex $v$ of $T$ is assigned a distinct  label;  no two vertices are assigned the same label;
			\item [(ii)] each vertex $v$ of $\{T_i\}_{i\in I}$ is assigned one of the labels  of the vertices of $T$ in the $\textrm{car}(v)$.
		\end{itemize}

	We say that $T_i$ is fully labeled, or completely labeled, if its vertices are labeled ${0,1,2, \ldots, n}$.	
	\end{defn}
	
Note that condition (ii) of the Sperner labeling is vacuously satisfied for any choice of labels of a simplex $T_i$  in the case when no vertex of $T_i$ lies on a face of $T$.
	
	\begin{thm}[Sperner's Lemma \cite{Sperner1928}] Every triangulation of $T$ labeled in a Sperner manner contains a completely labeled $T_i$.
	\end{thm}

\begin{figure}
	\includegraphics[width=.7\textwidth]{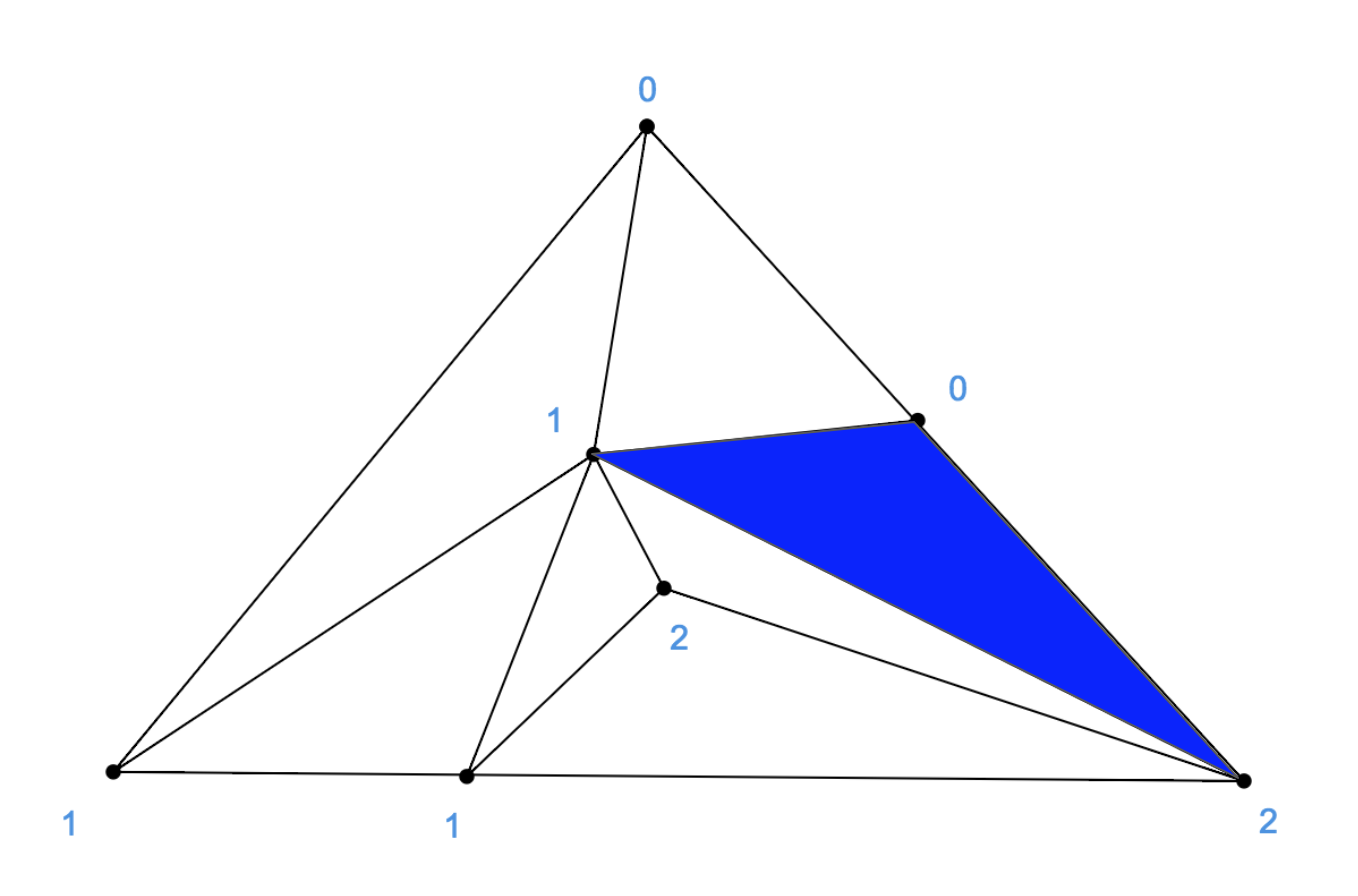}
	\caption{Example of Sperner's lemma.}
	\label{fig:sperner}
\end{figure}

See Figure~\ref{fig:sperner} for an example of Sperner's lemma in two dimensions. As mentioned, Sperner's lemma can be used to derive an elementary proof of the Brouwer fixed point theorem: any continuous map $f:B\to B$ has a fixed point, where $B$ is an $n$-dimensional closed ball. 

\begin{thm}[Brouwer's Fixed Point Theorem \cite{BurnsG2005}]
Let $T$ be an $n$-dimensional simplex and $f:T \mapsto T$ be a continuous function. Then there exists a point $x \in T$ such that $f(x)=x.$
\end{thm}
	
 We now introduce the notion of upper semicontinuity, which is essential for the formulation of Kakutani's fixed point theorem.
	
	\begin{defn}
		
		 A multivalued map $f :X \mapsto {2}^Y$ is upper semicontinuous if $f(x)$ is compact for all $x \in X$, and for every $x \in X$ and every $\epsilon > 0$, there exists a $\delta > 0$ such that if $z \in N_{\delta}(x)$  then $f(z) \subset N_{\epsilon}(f(x))$.
		
	\end{defn}
	
	Meaning that a multivalued map is upper semicontinuous if the image of every $x \in X$ is compact and (for every $\epsilon$ there exists a $\delta$ such that) the image of every element in the $\delta$ neighborhood around $x$ is a subset of the $\epsilon$ neighborhood around $f(x)$. Notice that this $\delta$ might not work for $x$ itself, meaning that, for the mentioned $\delta$, it is possible that $f(x) \not\subset N_{\epsilon}(f(z))$. There are a couple of other definitions for upper semicontinuity, equivalent to this one, but in this work we will use the above definition exclusively. See \cite{LM,YT} for the specifics of the other definitions.
	
	\begin{thm}[Kakutani's Fixed Point Theorem] Let $T$ be a $d$-dimensional simplex and let $f$ be an upper semicontinuous multivalued map from $T$ to nonempty, convex and compact subsets of $T$. Then, there exists a point $x \in T$ such that $x \in f(x)$.
	\end{thm}

	 As mentioned, the author of \cite{LM} tried a verity of ways of encoding the information of the multivalued maps onto the labels of a decomposition of $T$ so that Sperner's lemma can be used to prove Kakutani's theorem, but was ultimately unsuccessful. While the author of \cite{YT} was able to find such a labeling, the technicalities of the proof using the labeling are fairly complicated even in two dimensions and applying the proof to approximate the position of fixed points would be difficult.

\section{Polytope Version of Sperner's Lemma}\label{cubicle}

		In this section, we describe a more general Sperner lemma-type of result for convex polytopes (which shall be referred to as polytopes), following \cite{Bekker1995}. We use the result to prove another generalization of Sperner's lemma, Theorem \ref{thm:YS3}, one which can be directly used to prove Kakutani's theorem. For the decomposition of a polytope $P$ into polytopes $\{P_i\}_{i\in I}$ we assign a label $\{1, 2, \ldots, n+1\}$ to the vertices of $P_i$, such that no $(n-1)$-dimensional face of $P$ is labeled with a complete set of labels ($n+1$ different labels). The union of the labels of the vertices of the $P_i$'s contained on a face of $P$ together with the union of the labels of the vertices of that face of $P$ is also not allowed to be a complete set of labels. We refer to such a labeling as a non-degenerate labeling. The criteria that each vertex $v$ of $P$ be assigned a distinct label is also required, but the condition that two vertices of $P$ cannot have the same label will not be required. Notice, this labeling condition is much weaker than the original Sperner labeling condition.
		
			We introduce the following tool which will be used extensively in this work. Let $P$ be a convex $n$-dimensional  polytope. As before, consider the standard $n$-dimensional simplex $T=\textrm{conv}(0, e_1,e_2,\ldots ,e_n)\subset \mathbb{R}^n$. We denote the corresponding vertices of $T$ by $\{a_1,\ldots, a_{n+1}\}$. Given an $n$-dimensional, oriented, convex polytope  $P$, a  labeling  $\phi:P\to\{1,\ldots,n+1\}$, and the standard $n$-simplex $T$ with vertices $a_1,\ldots,a_{n+1}$, a \emph{realization} of $\phi$ is a continuous map $f: P \to T$, satisfying the following conditions:
			\begin{itemize}
				\item[(i)] If $v$ is a vertex of $P$ then $f(v)=a_{\phi(v)}$, i.e., $f(v)$ is the vertex $a_i$ of  $T$ with the index $i$ equal to the label of $v$;
				\item[(ii)] If $S$ is a face of $P$  with vertices $v_1,\ldots,v_k$, then $f(S)\subset \textrm{conv}(a_{\phi(v_1)},\ldots, a_{\phi(v_k)})$.	In general, in order to ensure that $f$ is continuous we map $S$ onto the entire convex hull of its labels.
			\end{itemize}
			Informally, a realization of $P$ is a continuous mapping of $P$ onto $T$ that `wraps' $\partial P$ around $\partial T$, such that the labels of the vertices of $P$ match with the indices $i$ of the vertices of $T$.  Such $f$ is in general non-injective. The reason $f: \partial P \to \partial T$ is because $T$ is a simplex and for any non-degenerate labeling of a face of $P$ with labels from the set $\{1,2,\ldots, n+1\}$ we will have that $f$ maps that face, under the conditions of a realization, to a subset of a face of $T$. Now, for a smooth, boundary preserving map $f:(M,\partial M)\to (N,\partial N)$ between two oriented $n$-dimensional manifolds with boundary, it is known that $\deg(f)=\deg(\partial f)$, where  $\partial f:\partial  M\to\partial N$ is the map induced by $f$ on the boundaries. The degree of the map $f$ can be defined as the signed number of preimages $f^{-1}(p)=\{q_1,\ldots,q_k\}$ of a regular value  $p$ of the map $f$, where each point $q_i$ is counted with a sign $\pm 1$ depending on whether  $df_{q_i}:T_{q_i}M\to T_{p}N$ is orientation preserving or orientation reversing. That is, $\deg(f)=\sum_{q\in f^{-1}(p)}\textrm{sign} (\det (df_{q}))$, where $p\in  N\setminus \partial N$ is a regular value of $f$. The definition of the Brouwer degree extends via homotopy to continuous maps. For us, the fact that $\deg(f)=\deg(\partial f)$ is incredibly useful because it is fairly easy to tell what $\deg(\partial f)$ is, under a realization of a labeling $\phi$, and $\deg(f)$ can tell us a lot about the existence of completely labeled simplices in a decomposition of $P$, as will be illustrated.

			\begin{prop}[\cite{Bekker1995}]\label{prop:Bekker}  Let $P$  be a convex  $n$-dimensional polytope.
				\begin{itemize}
					\item[(i)] Any labeling $\phi$ of some $P$ admits a realization $f$;
					\item[(ii)] Any two realizations of the same labeling are homotopic as maps of pairs $(P,\partial P)\mapsto (T,\partial T)$;
					\item[(iii)] If $\deg(f)=\deg(\partial f) \neq 0$ then $P$ is completely labeled, where $\deg(f)$ is the Brouwer degree of the realization of $\phi$, a labeling of $P$.
				\end{itemize}
			\end{prop}

			The following is a generalization of Sperner's lemma from \cite{Bekker1995}. We provide the proof for clarification and to ensure that the upcoming proof for Theorem \ref{thm:YS3} is intuitive.

			\begin{thm}[\cite{Bekker1995}]\label{thm:bekker3}
				Assume that $P$ is an $n$-dimensional polytope,  $P=\bigcup_{i\in I}P_i$ is a decomposition of $P$ into polytopes and $\phi:P \cup \bigcup_{i\in I}P_i\to\{1,\ldots,n+1\}$ is a non-degenerate labeling. Let $f$ be a realization of $\phi$, if $\deg(\partial f)\neq 0$ then there exists a polytope $P_i$ that is completely labeled. Also, if the decomposition is into simplicies and $\deg(\partial f)= 1$ then there exist an odd number of completely labeled simplicies $T_i$.
			\end{thm}
			
			\begin{proof}
				Let $f:P\to T$ be a realization of the labeling $\phi$; the existence of $f$ is ensured by Proposition \ref{prop:Bekker}.
				Suppose that no $(n-1)$-dimensional face of some $P_i$ has $(n+1)$ labels. We transform the polytope $P$ into another polytope $P^*$, homotopically transform the  decomposition  $P=\bigcup_{i\in I}P_i$ into another decomposition  $P^*=\bigcup_{i\in I}P^*_i$, and homotopically transform the map
				$f:P\to T$ into a map $f^*:P^*\to T$ as follows:
				\begin{itemize}
					\item We construct a new  polytope $P^*$ by appending  to the  vertices of $P$ all the vertices of the the $P_i$'s lying on the faces of $P$, and appending to the faces of $P$ all the faces of the $P_i$'s lying on the faces of $P$; the resulting polytope $P^*$  still has the decomposition $P^*=\bigcup_{i\in I}P_i$;
					\item We apply a homotopy deformation to the decomposition $\bigcup_{i\in I}P_i$ and to the realization $f$ to obtain a new decomposition  $P^*=\bigcup_{i\in I}P^*_i$, with the labeling of $P^*_i$ inherited from that of $P_i$, and a new realization  $f^*:P^*\to T$ so that, for every $i$, the restriction of $f^*_{\mid {P^*_i}}$ to $P^*_i$ maps $\partial P^*_i$ to $\partial T$ and is also a realization. The later property ensures that $\deg(f^*_{\mid {P^*_i}})=\deg(\partial f^*_{\mid {\partial P^*_i}})$.
				\end{itemize}

				From the non-degenerate condition which ensures that $f^*$ can be homotopically formed from $f$, we have $\deg(f^*) = \deg(f)$.	Using the addition and homotopy properties of the Brouwer degree, we obtain that \[\deg(\partial f) = \deg(f) = \deg(f^*) = \sum_i \deg(f^*_i) =  \sum_i \deg( \partial f^*_i)\]

				If $\deg(\partial f)\neq 0$, then $\sum_i \deg( \partial f^*_i)\neq 0$, which means that  there exists a polytope $P^*_i$ such that $\deg(\partial f^*_{\mid \partial P^*_i})\neq 0$, hence $\deg(f^*_{\mid P^*_i})\neq 0$. By Proposition \ref{prop:Bekker}, and since the labeling of $P_i$ is the same as the labeling of $P^*_i$, we have that $P_i$ is completely labeled. The intuition behind the proof is as follows. If $P_i$ maps onto $T$ under a realization $f^*$ it means that $P_i$ is completely labeled, otherwise, there would be a vertex of $T$ which would not be in the image of $f^*$. If there is some other realization $f$ of the labeling such that $\deg(\partial f)\neq 0$ then we can form $f^*$ such that $\deg(\partial f^*)\neq 0$, given that the labeling is non-degenerate, with the restriction of $f^*$ to every individual $P_i$ in the decomposition also  a realization. The non-degenerate condition ensures that $f^*$ can homotopically be formed from $f$ and so the degree for both is the same. Thus, there must be one $P_i$ which maps onto $T$ under $f^*$, meaning that $P_i$ is completly labeled. The continuity of $f^*$ follows from the fact that if two polytopes $P_i$ and $P_j$ share a $k$-dimensional face a mapping of that face, under a realization, can be the same for both $P_i$ and $P_j$  and so is continuous. It also follows from the above proof, using the orientation of $f$, that if $\deg(\partial f)= 1$ and the decomposition is into simplicies then there exist an odd number of completely labeled simplicies $T_i$. So, when $P$ is a simplex and $\deg(\partial f)= 1$ we have Sperner's lemma.
			\end{proof}
		
		\begin{figure}
			$\begin{array}{ccc}
			\includegraphics[width=0.3\textwidth]{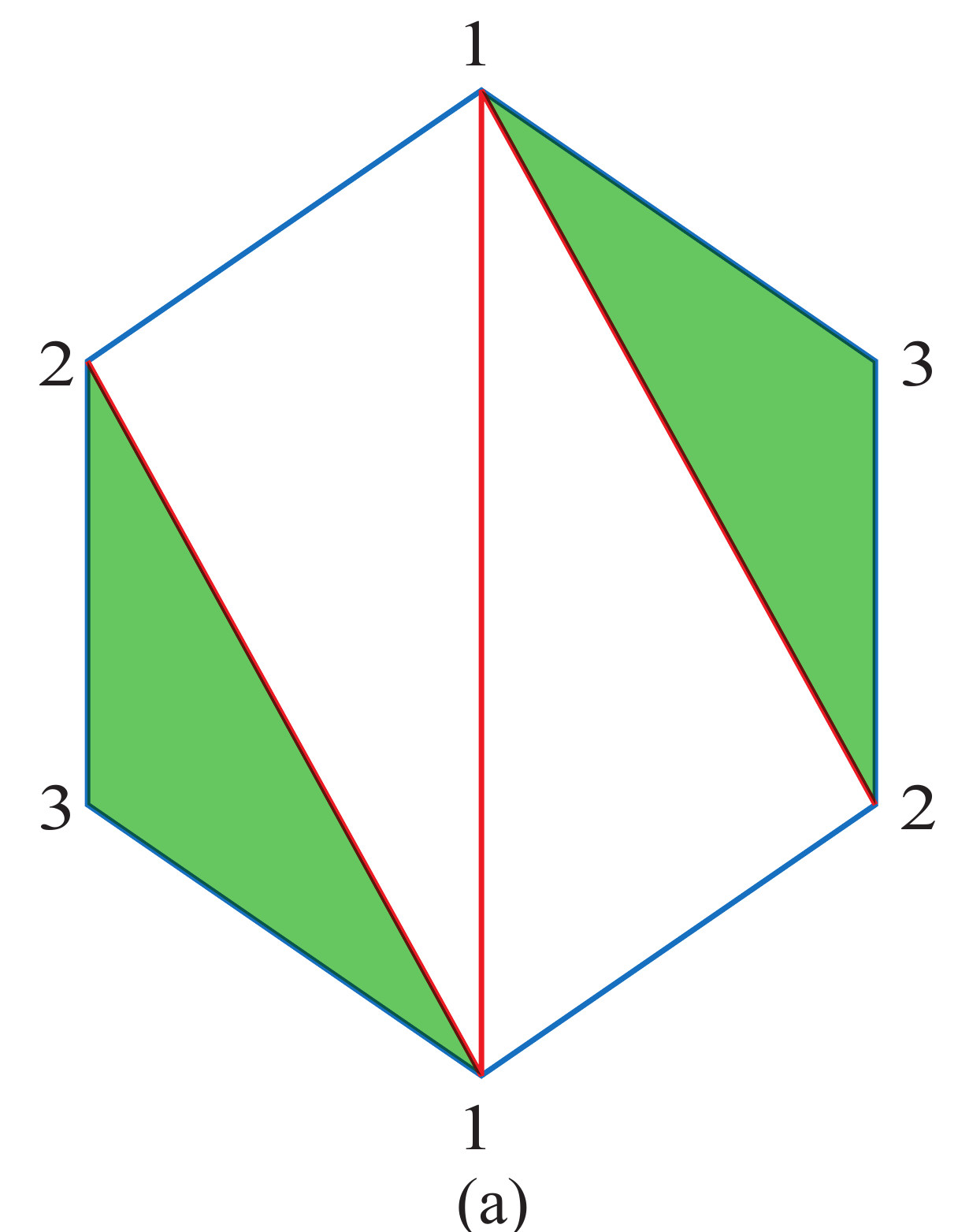}&
			\includegraphics[width=0.3\textwidth]{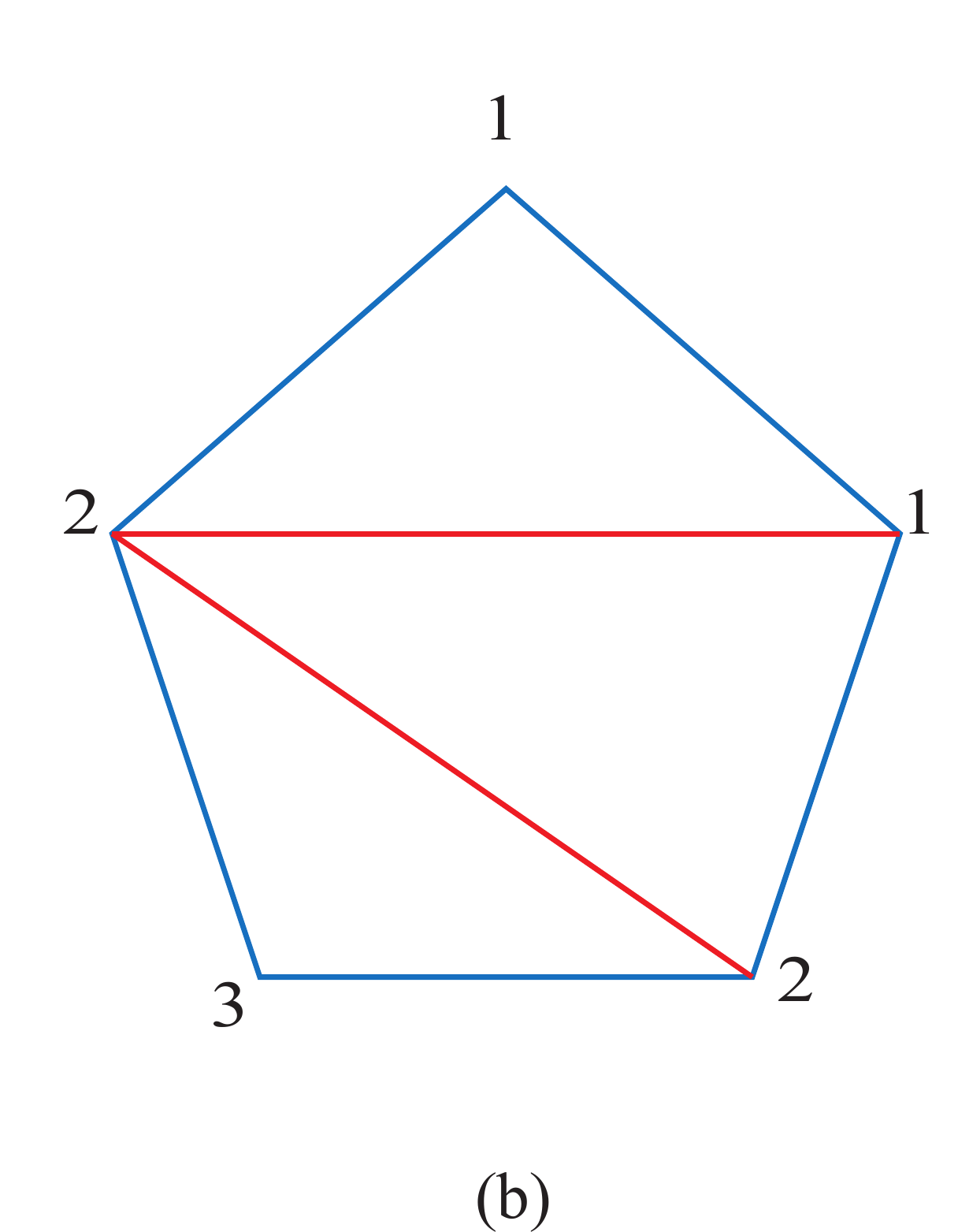}&
			\includegraphics[width=0.3\textwidth]{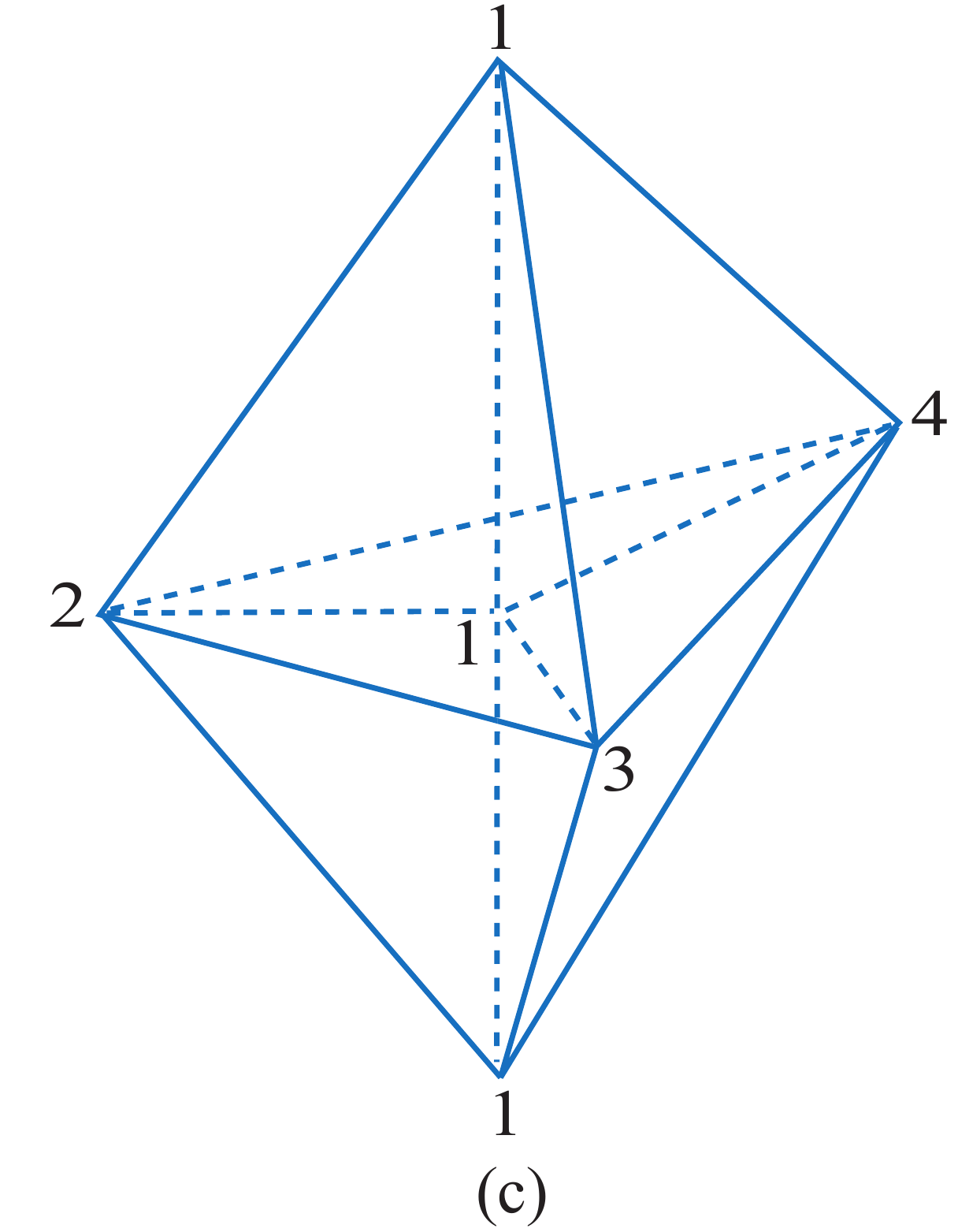}
			\end{array}$
			\caption{Examples of simplicial decomposition of polytopes and of labelings; refer to Example \ref{example}.}
			\label{hexagon}
		\end{figure}

		Note that in Theorem \ref{thm:bekker3} the assumption that $P$ is completely labeled alone is not sufficient to ensure that there exists a completely labeled simplex in the decomposition; the condition that $\deg(\partial f)\neq 0$  is necessary. See Example \ref{example}, (ii), (iii). It is also interesting to note that the non-degenerate condition of the labeling is incredibly weak, due to the fact that it only restricts the labeling of the $(n-1)$-dimensional faces of $P$. In \cite{Bekker1995} the condition is slightly stronger, that every $(r-1)$-dimensional face of $P$ cannot have $(r+1)$ different labels and every $(n-1)$-dimensional face of every $P_i$ in the decomposition cannot have $(n+1)$ different labels, however Theorem \ref{thm:bekker3} requires the much weaker condition that no $(n-1)$-dimensional face of $P$ has $(n+1)$ different labels both before and after the decomposition.

		\begin{ex}\label{example}
			
			(i) Consider the polygon $P$, the simplicial decomposition,   and the labeling shown in Figure \ref{hexagon}-(a). We have  $\deg(\partial f)=2$; there exists a completely labeled triangle.
			
			(ii) Consider the polygon $P'$, the simplicial decomposition,    and the labeling shown in Figure \ref{hexagon}-(b).   We have $\deg(\partial f)= 0$; there is no completely labeled simplex.
			
			(iii) Consider the polyhedron $P''$, the simplicial decomposition,    and the labeling shown in Figure \ref{hexagon}-(c). We have  $\deg(\partial f)= 0$;  there is no completely labeled simplex.
			
		\end{ex}

	We now give a more general notion of a realization. Given an $d$-dimensional, oriented, convex polytope  $P$, a  labeling  $\phi:P\to\{1,\ldots,n\}$, and a $d$-dimensional $n$-polytope $P'$ with vertices $a_1,\ldots,a_{n}$, a \emph{realization} of $\phi$ is a continuous map $f: P \to P'$, satisfying the following conditions:
	\begin{itemize}
		\item[(i)] If $v$ is a vertex of $P$ then $f(v)=a_{\phi(v)}$, i.e., $f(v)$ is the vertex $a_i$ of  $P'$ with the index $i$ equal to the label of $v$;
		\item[(ii)] If $S$ is an $(n-1)$-dimensional face of $P$  with vertices $v_1,\ldots,v_k$,  then $f(S)\subset \textrm{conv}(a_{\phi(v_1)},\ldots, a_{\phi(v_k)})$.
			\item[(iii)] $f(\partial P)=(\partial P')$,
	\end{itemize}
	
 We construct what shall be called $dP_N$ sets of vertices of a $d$-dimensional convex polytope, with $n$ vertices, as follows. If there exists a cross section of the polytope $P$ which keeps a set of vertices on the same side of the $d$-dimensional hyperplane (used during the cross section) then the set of those vertices forms a $dP_N$ set. With abuse of the definition, we also call a set of labels of those vertices a $dP_N$ set, with the restriction that every vertex be assigned a different label. Clearly, the set of a single vertex is a $dP_N$ set, for any $n$ and $d$. Note, however, that not all sets of vertices of a polytope form a $dP_N$ set. For the cube in Figure \ref{fig:crosssectionofcubes}, we see that $\{1,6,5,4\}$ and $\{3,2,7,8\}$ form such sets. For the same cube $C$, we can show that $\{3,7,8\}$ is an $3C_8$ set, but $\{5,4,3,2\}$ and $\{5,1,3\}$ are not. We consider the set of all vertices to not be a $dP_N$ set.

 Another way of thinking of $dC_N$ sets, relating specifically to cubes, is as follows. For any $ \mathbb{R}^d$ space, with each hyperoctant labeled with a label from a set of $2^d$ different labels, we take a $d$-dimensional hyperplane that passes through the origin. This hyperplane splits the space into two spaces $A$ and $B$, and it passes through some hyperoctants, and places the other hyperoctants entirely on one of its sides. For all hyperoctants that the hyperplane passes through, we make a single choice regarding the same $A$ or $B$ which we consider to contain them. Thus, $A$ and $B$ each contain a disjoint set of hyperoctants and the set of labels of the hyperoctants contained in either $A$ or $B$ forms a $dC_N$ set.

 As we will see, $(n-1)$-dimensional faces of a $P_i$ labeled with all labels from a $dP_N$ set map to $\partial P$, under a realization. Using the concept of a $dP_N$ set, we can add a new condition (iii) and we get the following theorem.

   \begin{figure}
	\includegraphics[width=.7\textwidth]{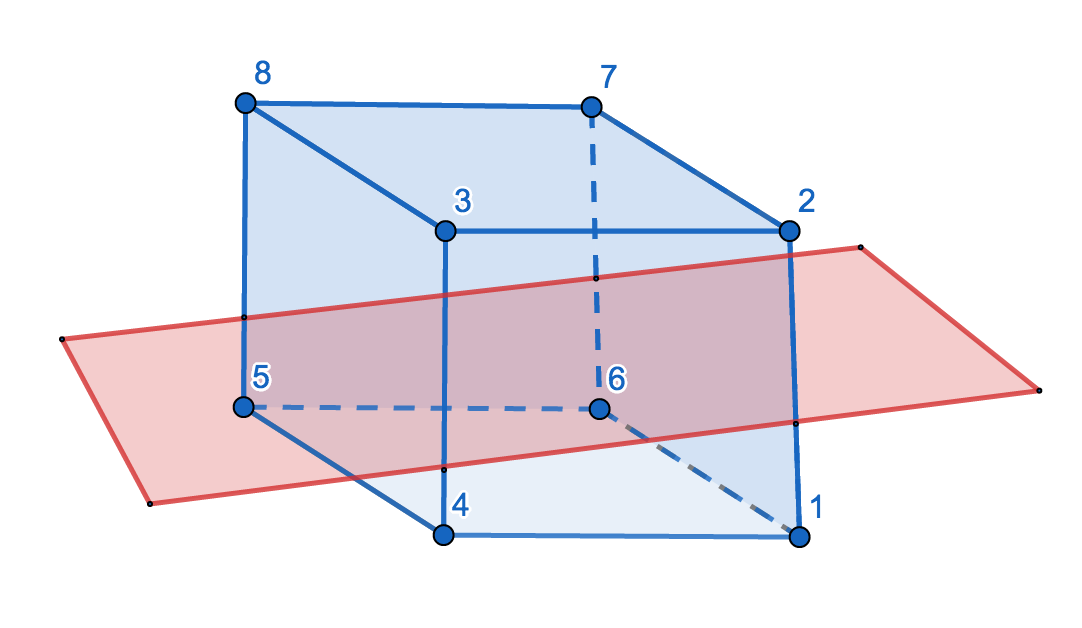}
	\caption{Cross section of cube.}
	\label{fig:crosssectionofcubes}
\end{figure}

\begin{thm}[Hyperplane Labeling Lemma]\label{thm:YS3} Let $P$ be a convex, $d$-dimensional polytope which is divided into convex, $d$-dimensional polytopes $\{P_i\}_{i\in I}$, with $I$ finite, such that $P=\bigcup_{i\in I}P_i$, and  for $i\neq j$, $\inte (P_i)\cap \inte (P_j)=\emptyset$ and $P_i\cap P_j$ is either empty or a face of both $P_i$ and $P_j$. Assume that $\phi:V(P)\cup\bigcup_{i\in I}V(P_i)\to\{1,2,\ldots,n\}$ is a labeling of the vertices of $P$ and of the $P_i$ such that:
	\begin{itemize}
		\item  There exists a $d$-dimensional $n$-polytope $P'$ with realization $f: P \to P'$ such that $\deg(\partial f)\neq 0$.
		\item Every vertex $v$ of every $P_i$ is assigned one of the labels of the vertices of $P$ in the $car(v)$. (We can replace this condition with the condition that the set of labels of the vertices contained on a $(n-1)$-dimensional face of $P$ is a $dP'_N$ set.)
		\item If the labeling of some $(d-1)$-dimensional face of some $P_i$ is not a $dP'_N$ set we add a finite number of points to the interior of the face and label them in such a way that the new labeling is a $dP'_N$ set.
	\end{itemize}
	
	A polytope $P_i$ is completely labeled if it has vertices or points (in the interior of a face) carrying labels $\{1,2,\ldots, n\}$. This lemma states that there is at least one completely labeled $P_i$. 
	
\end{thm}

	\begin{proof}
		We must show that there exists a realization $f: P\to P'$ such that $\partial P_i$ is always mapped to $\partial P'$ and the restriction of $f_{\mid {P_i}}$ to $P_i$ is also a realization. If the set of labels of the vertices of every $(d-1)$-dimensional face of every $P_i$ must be the same as a set $dP'_N$ of $P'$ then, under a realization, every $(d-1)$-dimensional face of a $P_i$ polytope will map to the boundary of $P'$ formed out of the vertices of $dP'_N$. We can split $\partial P'$ into two, using a hyperplane, and the $P_i$ face will map onto one of the halves. For one of the halves, every vertex is an image of the realization of the $(d-1)$-dimensional face of the $P_i$, otherwise the set of labels of the face won't form a $dP'_N$ set. Thus $\partial P_i$ is always mapped to $\partial P'$. While it is easy to show that every $(d-1)$-dimensional face of every $P_i$ will map to the boundary of $P'$, it is slightly more complicated to show that such a map can be made continuous. We must prove that for two $(n-1)$-dimensional faces $S_i$ and $S_j$ sharing a $k$-dimensional face $W$ there always exists a continuous realization $f^*$ of $S_i$ and $S_j$ which maps $W$ onto a fixed subset of $f^*(S_i)\cap f^*(S_j)$. We start with the hyperplanes $H_i$ and $H_j$ which provide the $dP_N$ sets of $S_i$ and $S_j$. Because they share the face $W$ there will be a path connected subset of $\partial P$ which both $H_i$ and $H_j$ keep on their sides corresponding to their $dP_N$ set and $W$ can be mapped to that subset. The remainder of the proof is the same as in Theorem \ref{thm:bekker3}.
		\end{proof}

\begin{rem}
	When proving Kakutani's fixed point theorem we will have that $P$ is a cube with $P'=P$ and with every vertex of $P$ labeled with a unique label from $\{1,\ldots,n\}$ such that the $\deg(\partial f)= 1$. 
\end{rem}

The new condition (iii) is not very limiting. First, notice that it only applies to $(d-1)$-dimensional faces of the $P_i$'s and not every dimensional face of the $P_i$'s. Furthermore, we are allowed to add vertices and labels to the $(d-1)$-dimensional faces to fix their labels without the labels effecting any of the other faces of the $P_i$. Obviously, if a $P_i$ has a face on the boundary of $P$ and we add a vertex to the face, thereby adding a vertex to the original $P$ (for the purposes of the realization), we must make sure that it satisfies condition (ii) of the labeling. In Section \ref{kakutani}, we will use the hyperplane labeling lemma to prove Kakutani's theorem, but just the fact that if we were to fix every face we would get a completely labeled cube will be enough to show that a fixed point exists, and that one $(d-1)$ face of some $P_i$, with labels that don't form a $dC_N$ set, will approximate its location. We actually never need to add vertices to the decomposition. For an example of Theorem \ref{thm:YS3}, see Figures \ref{fig:acube}, \ref{fig:labeling}, \ref{fig:problematicface}, and \ref{fig:addingvertexandCL}.

  In Figure \ref{hpllexample}, the cube is decomposed into two polytopes, with a face of both labeled with $\{1,4,3,7,6,5\}$, which is not a $dC_N$ set. We don't have a completely labeled polytope in the decomposition. As one can see, we can decompose the face into two faces, as in Figure \ref{hpllface3}, but every such decomposition will yield a face labeled with labels which don't form a $dC_N$ set.

\begin{figure}
	\includegraphics[width=.7\textwidth]{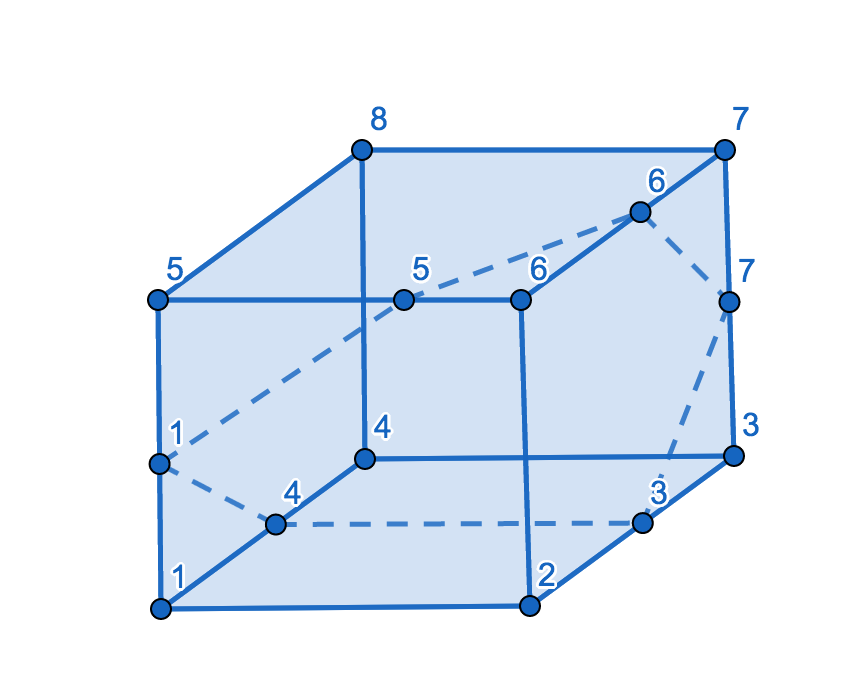}
	\caption{Decomposition of the cube into two polytopes.}
	\label{hpllexample}
\end{figure}

\begin{figure}
	\includegraphics[width=.7\textwidth]{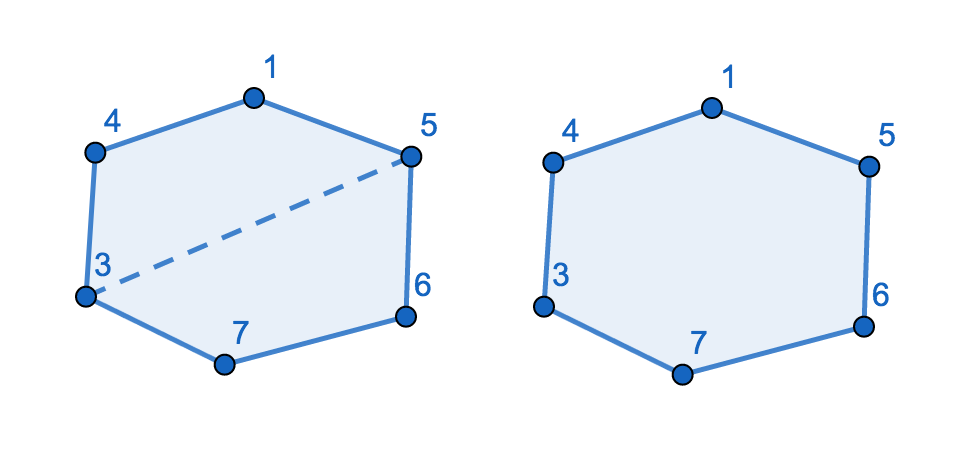}
	\caption{We decompose the face on the right into two faces not labeled with $dC_N$ sets.}
	\label{hpllface3}
\end{figure}

\begin{rem}
	The following condition can take the place of the third condition. We call two labels $l$ and $w$ similar, or rather $k$-similar, if there is a $k$-dimensional face of $P'$ such that $a_l$ and $a_w$ share the face. Also, a label is similar to itself for all dimensions. If a vertex of $P_i$  shares a $k$-face with another vertex (either of $P_i$ or of $P_j$) then both of their labels must be $k$-similar, for all $k$.
\end{rem}

	\section{Proof of Kakutani's Fixed Point Theorem}\label{kakutani}
	
			As mentioned, the hyperplane labeling lemma can be used to give an alternative proof of Kakutani's theorem. We  consider a (multivalued) upper semicontinuous map $f:C\to C$ on a $d$-dimensional cube $C$ whose $n=2^d$  vertices are defined by $z_{i_1,\ldots,i_d}=(\pm 1,\ldots,\pm 1)\in\mathbb{R}^d$. We also require that $f(z)$ be convex for all $z$. Let $\mathscr{Z}=\{\pm 1\}$. For a point $z=(x_1,\ldots, x_n)$ in $\R^n$ and $\Delta z=(x'_1-x_1,\ldots, x'_n-x_n)$, for every element $z'=(x'_1,\ldots, x'_n) \in f(z)$, we can assign a $O$ label $\ell\in\mathscr{Z}^n$, based on the following conditions.
		\begin{description}
			\item[Condition O]{$ $}
			\begin{itemize}\item if $\Delta z\in\inte (\mathscr{O}_\ell)$ then we assign to $z$ the label  $\ell$;
				\item if $\Delta z\in \inte(\mathscr{O}_{\ell_1}\cap \ldots
				\cap \mathscr{O}_{\ell_k})$ for some labels $\ell_1,\ldots,\ell_k$ that are mutually distinct, then we assign to $z$ either one of the labels $\ell_1,\ldots,\ell_k$, say $\ell_i$, contingent on the condition that $\Delta z\in\inte (\mathscr{O}_{\ell_i})$ for slight homotopy of $f$.
			\end{itemize}
		\end{description}
		
		Now, for a point $z=(x_1,\ldots, x_n)$, we can also assign any $H$ label $\ell\in\mathscr{Z}^n$ as follows:
		
		\begin{description}
			\item[Condition H]{$ $}
			\begin{itemize}
				\item  consider the translation of $\R^d$ so that $z$ is at the origin;
				\item  there is a hyperplane dividing $\R^d$ and passing through $z$, after the translation, so that $f(z)$ is contained on one of its sides;
				\item  $\ell \in dC_N$ of $\R^d$ created by the side of the hyperplane which contains $f(z)$;
				\item  all labels in this $dC_N$ set must satisfy the second condition of the hyperplane labeling lemma;
				\item the set of labels that $z$ can be labeled with under Condition $O$ must be some subset of this $dC_N$ set.
			\end{itemize}
		\end{description}
		
		We can show that such a labeling always exists. First, $z$ can be labeled with labels from potentially multiple sets $dC_N$, corresponding to every hyperplane that satisfies the other criteria of the condition. Second, for every hyperplane there are two $dC_N$ sets corresponding to the side that $f(z)$ is on, one in which the hyperoctants that the hyperplane passes through belong to the $dC_N$ set and the other in which the hyperoctants that the hyperplane passes through do not belong to the $dC_N$ set. This choice is what allows us to always satisfy the second condition of the hyperplane labeling lemma.

	\begin{thm}[Kakutani's Fixed Point Theorem]\label{5} Let C be a $d$-dimensional closed cube and let $f$ be an upper semicontinuous multivalued mapping from $C$ to nonempty, convex and compact subsets of C. Then, there exists a point $x \in C$ such that $x \in f(x)$.
\end{thm}

	\begin{proof}
Let $\{C^N_i\}$  be a subdivision of $C$, with $\textrm{diam}(C^N_i)\to 0$ as $N\to\infty$. We label vertices of the subdivision based on Condition $O$. Unless there is some $z^*$ such that $ z^*\in f(z^*)$, we always have an infinite $N$ such that we can add a finite number of vertices to every face of every $C_i$ and label the points, based on Condition $H$, in such a manner that the third condition of the hyperplane labeling lemma is satisfied.
	We prove this by contradiction. Suppose otherwise, then there are infinite $N$ with a $C_i$ such that the set of labels of all points on a face of some $C_i$ does not form a $dC_N$ set. By compactness and the fact that $\textrm{diam}(C^N_i)\to 0$, the sequence of such faces contains a subsequence which converges to a point $z'$. Take $\mathscr{O}_\ell$, for various $\ell$, to be the smallest set of hyperoctants that contain $f(z')$ in their interior. By upper semicontinuity of $f$, we can find $\delta$ such that for $x \in N_\delta(z')$ we have that $f(x)\subset N_{\epsilon}f(z') \subset \bigcup_\ell \mathscr{O}_\ell$. Take $N$ big enough so that the problematic face $F_i$ of $C_i$ is contained within $N_\delta(z')$. Thus, there is a hyperplane, corresponding to the $dC_N$ set of $z'$, which places $F_i$ on one side and $f(F_i)$ on the other, and we can add a finite number of points to the face and label, base on Condition $H$, so that the set of labels on the face is a $dC_N$ set. Notice, we only use Condition $H$ to fix problematic faces. Furthermore, when fixing faces $F_i$ we use labels from the $z'$ which the faces converge to and for each point $z'$ with problematic faces converging to it we choose one hyperplane $H$ to use for fixing the faces. More formally, for $N_{\epsilon}f(z')$ and $N_\delta (z')$ with a hyperplane separating the two, if $F_i \subset N_\delta (z')$ then we can fix it with the labels of $z'$, under Condition $H$, for that specific hyperplane. If we have two such points $z'$ then we choose the one closest to the center of $F_i$. If there are two different hyperplanes that work for $z'$, under Condition $H$, we choose one and fix all problematic faces that converge to $z'$ with that specific hyperplane.
	
		Let $\{C^N_i\}$  be a subdivision of $C$, with $\textrm{diam}(C^N_i)\to 0$ as $N\to\infty$. We label every vertex $z$ in the subdivision based on Condition $O$. Such a labeling satisfies the first two conditions of the hyperplane labeling lemma. Now for infinite $N$, we can add a finite number of vertices to every face of every $C_i$ and label the points, based on Condition $H$, in such a manner that the third condition of the hyperplane labeling lemma is satisfied. Thus, we can form a fine enough decomposition such that the three labeling conditions for Theorem \ref{thm:YS3} are satisfied. So we have a completely labeled cube $C^N_{i^*}$ for infinite subdivisions of $C$, with $\textrm{diam}(C^N_i)\to 0$ as $N\to\infty$. By compactness, we can find a converging sequence of such cubes. Suppose they converge to a point $x^*$ which is not a fixed point.  If for infinite $N$ no face of every $C^N_{i^*}$ was fixed then clearly $x^*$ is a fixed point. So we must have a sequence of problematic faces $F_i$ which were fixed so that the third condition of the hyperplane labeling lemma is satisfied. Because the faces converge to $x^*$ they are fixed with the labels of $x^*$ under Conditions $H$. Notice, however, that the set of such labels is not all $2^n$ labels and has the Condition $O$ labels as a subset. Thus,  $C^N_{i^*}$ cannot be completely labeled and we have a contradiction.

\end{proof}

\begin{cor}\label{4.4}
	Let C be a $d$-dimensional closed cube and let $f$ be an upper semicontinuous multivalued mapping from $C$ to nonempty, convex and compact subsets of C. Let $\{C^N_i\}$  be a subdivision of $C$, with $\textrm{diam}(C^N_i)\to 0$ as $N\to\infty$. If we label vertices of $C^N_i$ based on Condition $O$ then either some sequence of problematic faces $F_i$ (faces not labeled with $dC_N$ labels) converge to a fixed point of $C$ or completely labeled cubes converge to the fixed point.
\end{cor}

It has recently been shown in \cite{JEANJACQUESHERINGS200889} that locally gross direction preserving maps, which are possibility discontinuous, have the fixed point property. The definition of locally gross direction preserving maps for single valued maps is as follows. 

\begin{defn}\label{lgdp} A locally gross direction preserving map is a  map $f:P \to P$ such that for all $x \in P$, where $x$ is not a fixed point, there exists a $\delta$ such that for $y,z\in N_\delta (x) \cap P$ we have that \newline
	
	$(f(y)-y)^T(f(z)-z)\geq 0$. \newline

\end{defn}

Meaning that for some ball around $x$ every point in the ball should be moving in the same general direction and so the dot product between the two vectors should be non-negative. Clearly every continuous function is locally gross direction preserving, however it is also clear that some locally gross direction preserving maps might have discontinuities. The requirement that points should be moving, under the map, in a general direction does not ensure that the map be continuous. It was shown in \cite{JEANJACQUESHERINGS200889}  that every locally gross direction preserving map from a convex polytope to itself has a fixed point. The results were then extended in \cite{lmk} to all compact and convex subsets of $R^n$. It is straightforward to see that the hyperplane labeling lemma provides a constructive proof that every locally gross direction preserving map from a convex polytope to itself has a fixed point. We start by giving a more general definition of locally gross direction preserving maps, one which works for multivalued mappings from $P$ to nonempty, compact subsets of $P$. 

\begin{defn}\label{H}
	We call $f$, a multivalued map, locally gross direction preserving if for all $x$ which is not a fixed point there exists some $\delta$ such that $N_\delta (x)$ and $f(N_\delta (x))$ can be separated by some hyperplane.
\end{defn} 

\begin{thm}  
	
Definition \ref{lgdp} and Definition \ref{H} are equivalent.   
	
\end{thm}           

\begin{proof}
	If a single valued map is locally gross direction preserving, under Definition \ref{H},  then for all $x$ which are not fixed points there exists some $\delta$ such that $N_\delta (x)$ and $f(N_\delta (x))$ can be separated by some hyperplane. Thus for any two points $y,z \in N_\delta (x)$ we will have that $(f(y)-y)$ and $(f(z)-z)$ are contained on the same side of the hyperplane and so $(f(y)-y)^T(f(z)-z)\geq 0$. The converse follows by a similar argument.   
\end{proof}

	\begin{thm}[Generlization of Kakutani's Fixed Point Theorem] Let $C$ be a d-dimensional convex cube and let $f$ be an multivalued  locally gross direction preserving map from $C$ to nonempty, compact subsets of $C$. Then, there exists a point $x \in C$ such that $x \in f(x)$.
	\end{thm}
	
	\begin{proof}
		The proof for this generalization can be taken almost directly from the previous proof. Notice, the assumption that $f(x)$ be a convex set is not necessary.
	\end{proof}

\section{Conclusion}

We see that there is a combinatorial lemma which provides a simple proof for Kakutani's theorem and can be used to find fixed points of multivalued maps. As mentioned, in practice one need not fix the problematic faces in the decomposition, for either completely labeled cubes or a sequence of such faces approximate the location of the fixed point. We end with a statement about the following generalization of Kakutani's theorem. Suppose $X$ and $Y$ are two compact manifolds which are subspaces of $\R^d$ and we have a continuous point valued map $f:X \to Y$. It is well known that every such map has a fixed point index, meaning an index that can help determine whether $f$ has a fixed point. We start by looking at the function $g: \partial X \to S$, where $S$ is the unit sphere and $g(x)=\dfrac{f(x) -x}{\Vert f(x) -x\Vert}$. If we have that the Brouwer degree of $g$ is not zero then $f$ must have at least one fixed point. Intuitively, if $X$ is a polytope (including a polytope with holes), the Brouwer degree of $g$ can be approximated by the Brouwer degree of the realization of a decomposition of $X$ into $d$-dimensional quadrilaterals $Q^N_i$ labeled based on Condition $O$. Thus, Theorem \ref{thm:bekker3} can probably be used to provide a combinatorial proof that the fixed point index works. In \cite{GS1} such a proof was given for the case of correctly aligned windows, but the proof can be extended to various cases where the Brouwer degree of $g$ does not equal zero. It makes sense that such an index should also work for upper semicontinuous multivalued mappings, where each point maps to a convex set. More specifically, suppose we have a upper semicontinuous multivalued map $h:X \to 2^Y$, from points in $X$ to convex subsets of $Y$, and for every continuous function $f:X \to Y$ we have that the Brouwer degree of $g$ is not zero then presumably there is $x\in X$ with $x \in h(x)$. One way of proving such a theorem would be to use von Neumann's approximation lemma to approximate $h$ with a sequence of continuous functions $\{f_n\}$ and follow the method found in \cite{LM}. It might be useful, though, to find a combinatorial proof for such a theorem and the hyperplane labeling lemma can probably be used in such a manner, although certain restrictions would have to be placed on $\partial X$ so that the second condition of the lemma be satisfied.

	\section*{Acknowledgement}

I am grateful to Marian Gidea for making me aware of numerical applications of Kakutani's fixed point theorem, for insightful discussions and for reading and commenting on an earlier draft of this work.

	\begin{figure}[h]
	\includegraphics[width=.5\textwidth]{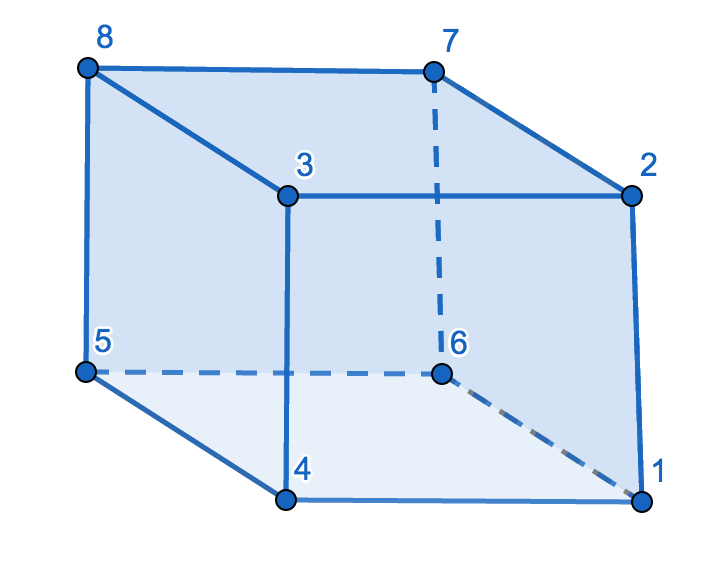}
	\caption{We start with a cube which we decompose into smaller cubes.}
	\label{fig:acube}
\end{figure}

\begin{figure}[h]
	\includegraphics[width=.5\textwidth]{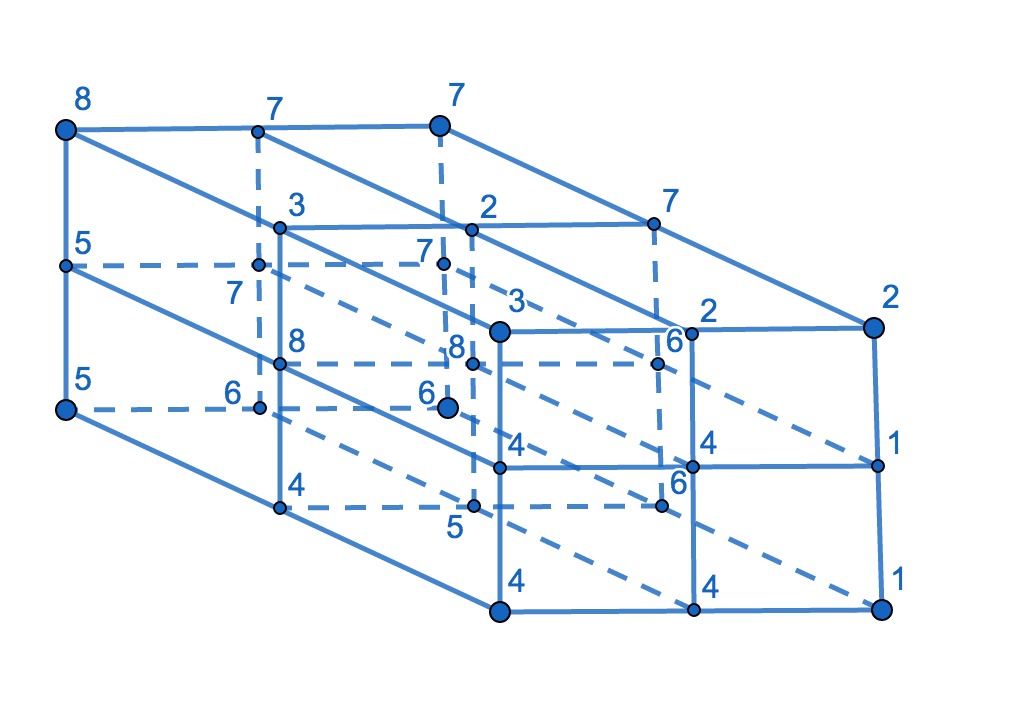}
	\caption{We label the decomposition, ensuring that condition (i) and (ii) are satisfied. The realization is formed from the cube to itself.}
	\label{fig:labeling}
\end{figure}

\begin{figure}[h]
	\includegraphics[width=.5\textwidth]{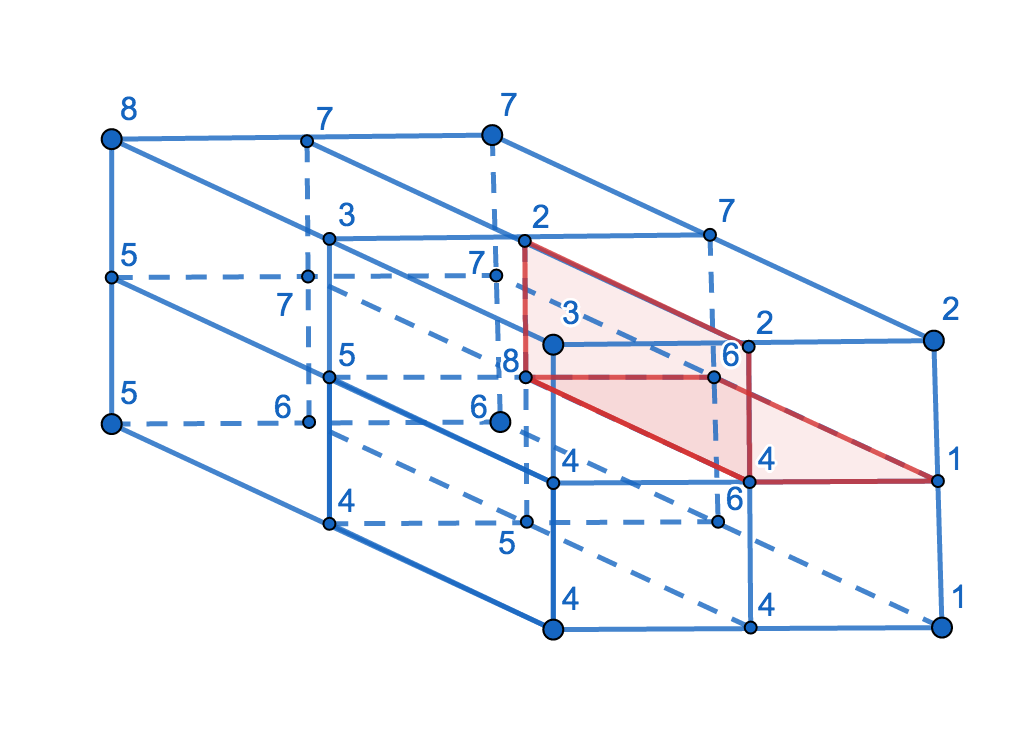}
	\caption{The following are the only problematic $(n-1)$-faces which don't satisfy condition (iii).}
	\label{fig:problematicface}
\end{figure}

\begin{figure}[h]
	\includegraphics[width=.5\textwidth]{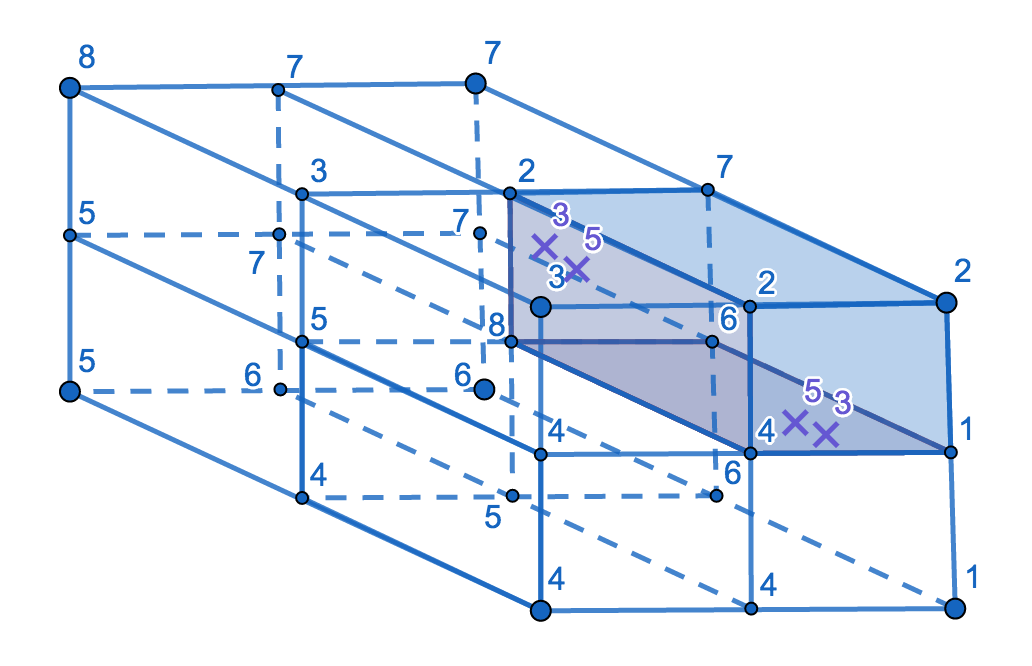}
	\caption{We fix the faces and a completely labeled cube appears. Notice, for the face labeled \{2,2,8,4\} adding a $3$ alone would have also been enough to satisfy condition  (iii), but either way a completely labeled cube appears.}
	\label{fig:addingvertexandCL}
\end{figure}

\clearpage

\bibliographystyle{alpha}

\bibliography{kakutani_biblio}

\end{document}